\documentclass[12pt,a4paper]{article}

\usepackage[margin=1in]{geometry}
\usepackage{amsmath,amssymb,amsthm}
\usepackage{enumitem}
\usepackage{hyperref}
\usepackage{tikz}

\newtheorem{proposition}{Proposition}
\newtheorem{corollary}{Corollary}
\newtheorem{theorem}{Theorem}

\newtheorem{remark}{Remark}
\newtheorem{lemma}{Lemma}

\providecommand{\keywords}[1]{\medskip\noindent\textbf{Keywords:} #1}
\usepackage{xcolor}
\usepackage{comment}
\newcommand{\red}{\color{red}}

\newcommand{\lam}{\lambda_{\max}}
\newcommand{\RT}{R_T}

\title{Spectral Monotonicity under Leaf Attachment and Limiting Behavior in Discrete Einstein Trees}
\author{
Shuliang Bai\thanks{Beijing Yanqi Lake Institute of Mathematical Sciences and Applications, China. Email: \texttt{baishuliang@bimsa.cn}}
\and
Haoxuan Cheng\thanks{School of Mathematical Sciences, Fudan University, Shanghai 200433, China. Email: \texttt{hxcheng25@m.fudan.edu.cn}, \quad \textit{Corresponding author}}
\and
Bobo Hua\thanks{School of Mathematical Sciences, LMNS, Fudan University, Shanghai 200433, China. Email: \texttt{bobohua@fudan.edu.cn}}
}

\date{\today}

\begin{document}
\maketitle
\begin{abstract}
Let $R_T$ be the Ricci matrix of a finite tree $T$ introduced in \cite{BaiChengHua2026}, the largest eigenvalue $\lambda_{\max}(R_T)$ determines the sign of a discrete Einstein metric curvature on the tree. This paper investigates the asymptotic behavior of the sequence $\lambda_k = \lambda_{\max}(R_{T_k})$ obtained by repeatedly adding pendant edges at a fixed vertex. We prove that $\lambda_k$ converges to a limit $\lambda_\infty$ that depends only on the local branch data of $T$, and establish a first-order asymptotic expansion:
\[
\lambda_k = \lambda_\infty + \frac{\alpha}{d+k} + O\!\left(\frac{1}{(d+k)^2}\right),
\]
where $d$ is the degree of the original vertex, and the coefficient $\alpha$ is given by a spectral projection. As a corollary, when $\alpha \neq 0$, $\lambda_k$ is eventually strictly monotonic (increasing or decreasing). This theory reveals the fine influence of local leaf addition on the global spectrum.
\end{abstract}


\section{Introduction}

Let $T=(V,E)$ be a finite tree. We \cite{BaiChengHua2026} introduced the \emph{Ricci matrix} $R_T\in\mathbb{R}^{E\times E}$ of a tree, defined by
\[
(R_T)_{e,e'} =
\begin{cases}
-\bigl(\frac{1}{d_x}+\frac{1}{d_y}\bigr), & e=e'=\{x,y\},\\[4pt]
\frac{1}{d_z}, & e\neq e',\ e\cap e'=\{z\},\\[4pt]
0, & e\cap e'=\varnothing,
\end{cases}
\]
where $d_v$ is the degree of vertex $v$. A key property of this matrix is that a discrete Einstein metric on a tree (i.e., edge weights with constant Lin--Lu--Yau curvature \cite{LinLuYau2011}, in the sense of optimal-transport Ricci curvature on graphs \cite{Ollivier2009}) corresponds exactly to the Perron eigenvector of $R_T$, with the Einstein curvature given by $\kappa = -\lambda_{\max}(R_T)$.

Consequently, the sign of $\lambda_{\max}(R_T)$ entirely determines the sign of the curvature, and the study of its magnitude and variation is of significant geometric and dynamical interest.

We\cite{BaiChengHua2026}  proved that 
$\lambda_{\max}(R_T)\le 0$ implies $T$ must be a caterpillar (i.e., a tree whose 
leaves can be removed to leave a path); they also established monotonicity results for leaf 
attachment: attaching a leaf at a vertex of degree $\le 2$ does not decrease 
$\lambda_{\max}$; when $\lambda_{\max}(R_T)<0$, leaf attachment at any vertex 
strictly increases $\lambda_{\max}$.
Building on this foundation, Cheng\cite{ChengHua2026} provided a complete 
classification of trees with $\lambda_{\max}(R_T)\le 0$, proving that for spine length 
$m\ge 12$ negativity occurs exactly for endpoint families $C_m(a,b)$ with 
$1\le a,b\le 3$, $(a,b)\neq(3,3)$, and giving explicit finite lists for shorter 
spines. Their classification also confirmed that the only infinite zero-curvature 
family is precisely $(3,0,\ldots,0,3)$ discovered in \cite{BaiChengHua2026}. These phenomenon naturally raises the following question:

\begin{center}
\emph{When repeatedly adding pendant edges at a fixed vertex, how does $\lambda_{\max}(R_T)$ change? Does the limit exist? When is it eventually monotonic?}
\end{center}

This paper provides a systematic answer to this question. Consider a fixed tree $T$ and a vertex $v$ on it. Let $T_k$ be the tree obtained by adding $k$ additional pendant edges at $v$. Set $\lambda_k = \lambda_{\max}(R_{T_k})$. By reducing $R_{T_k}$ on an orbit-equivariant subspace, we express $\lambda_k$ as the largest eigenvalue of a family of finite-dimensional matrices
\[
Q_k = Q_\infty + \frac{1}{d+k}B,
\]
where $d=d_T(v)$, $Q_\infty$ is a block upper-triangular matrix whose diagonal blocks consist of the Dirichlet truncation matrices $A_j$ from each branch after removing $v$ and a zero scalar block. See Section~\ref{sec:limit} for accurate definitions of $A_j$ and $B$. From this, we obtain the following results:

\begin{itemize}
\item \textbf{Limit Formula} (Proposition~\ref{prop:limit}): The sequence $\lambda_k$ converges to a limit
\[
\lim_{k\to\infty}\lambda_k = \lambda_\infty := \max\bigl(0,\lambda_{\max}(A_1),\ldots,\lambda_{\max}(A_d)\bigr).
\]

\item \textbf{First-Order Asymptotic Expansion} (Proposition~\ref{prop:asym}):
 When the limit $\lambda_\infty$ is simple, the convergence rate is governed by
\[
\lambda_k = \lambda_\infty + \frac{\alpha}{d+k} + O\!\left(\frac{1}{(d+k)^2}\right),
\]
where $\alpha$ is a constant determined by the local geometry of $T$ at $v$. 
In particular, the sign of $\alpha$ controls whether $\lambda_k$ approaches $\lambda_\infty$ 
from above or below, and guarantees eventual monotonicity.

\item \textbf{Eventual Tail Monotonicity} (Theorem~\ref{thm:tail-degenerate}):
If $\alpha \neq 0$, then for sufficiently large $k$, the sequence $\{\lambda_k\}$ is strictly monotonic: strictly decreasing towards $\lambda_\infty$ if $\alpha>0$, and strictly increasing towards $\lambda_\infty$ if $\alpha<0$.
\end{itemize}


The paper is organized as follows: Section~\ref{sec:notation} reviews the Ricci matrix and basic notation; Section~\ref{sec:one-step} derives the Rayleigh difference formula for one-step leaf addition and the $\rho_v$-sharp criterion, providing a coarse degree-dependent threshold; Section~\ref{sec:limit} establishes the block structure of $Q_k$ and proves the limit formula; Section~\ref{sec:monotonicity} uses analytic perturbation theory to prove the first-order expansion and eventual tail monotonicity; Section~\ref{sec:example} presents  numerical examples to demonstrate the theory.

\subsection*{Relation to Existing Work}
Here we briefly compare the leaf-attachment 
behavior of the Ricci matrix with that of other classical graph matrices.

For the adjacency matrix $A(G)$, adding a pendant vertex at a vertex $v$ strictly increases the spectral radius 
regardless of $\deg(v)$; this follows from the Perron--Frobenius theorem and the interlacing inequality 
(see, e.g., \cite[Chapter~2]{Stevanovic2014}). For the graph Laplacian $L(G)$, the largest eigenvalue is 
bounded by the maximum degree, and attaching a leaf typically increases it, though the exact increment depends 
on the graph structure; the algebraic connectivity (second smallest eigenvalue) may increase or decrease 
depending on the attachment vertex \cite{Birolutal2008}. For the normalized Laplacian $\mathcal{L}(G)$, the 
effect of attaching a leaf is more complex and generally non-monotone \cite{Butler2007}.

The Ricci matrix $R_T$ exhibits a richer phenomenology. The present paper reveals an additional layer: 
when leaves are repeatedly added at a fixed vertex, $\lambda_{\max}(R_{T_k})$ converges 
to a limit $\lambda_\infty$ determined by the branches of $T$, and the tail behavior 
is governed by a first-order coefficient $\alpha$ whose sign dictates eventual 
monotonicity. This combination of convergence and sign-dependent monotonicity 
appears to be a distinctive feature of the Ricci matrix, not shared by the adjacency 
or Laplacian matrices.

\section{Notation}\label{sec:notation}
Let $T=(V,E)$ be a finite tree. For a vertex $v$,  let $d_v$ be its degree. We call an edge $ab=\{x,y\}\in E(T)$ \emph{internal} if both endpoints have
degree at least two, i.e., $d_x\ge 2$ and $d_y\ge 2$; equivalently, neither
endpoint is a leaf of $T$.
For an edge $e=\{x,y\}$, the \emph{Lin–Lu–Yau Ricci curvature} of a weighted tree 
with edge weights $w$ is given by \cite{BaiChengHua2026}
\[
\kappa_{xy} = -\left( \frac{S_x-2w_{xy}}{w_{xy}d_x} + \frac{S_y-2w_{xy}}{w_{xy}d_y} \right),
\qquad S_v = \sum_{u\sim v} w_{uv}.
\]

Denote its Ricci matrix by $\RT$, defined as
\[
(\RT)_{e,e'}
=
\begin{cases}
-\left(\dfrac1{d_x}+\dfrac1{d_y}\right), & e=e'=\{x,y\},\\[4pt]
\dfrac1{d_z}, & e\neq e',\ e\cap e'=\{z\},\\[4pt]
0, & e\cap e'=\varnothing.
\end{cases}
\]

A tree is called \emph{discrete Einstein} if there exists a positive weight function
$w$ such that $\kappa_{xy}\equiv\kappa$ is constant. 
In \cite{BaiChengHua2026}, we proved that this is equivalent to $R_T w = \lambda w$
with $\lambda = -\kappa$, and that $\lambda_{\max}(R_T)$ is simple with a positive
eigenvector (the Perron vector). Hence $\kappa = -\lambda_{\max}(R_T)$.

\begin{remark}[Schr\"odinger-operator form of $R_T$]\label{rem:schr}\cite{BaiChengHua2026}
The Ricci matrix can be regarded as a weighted Schr\"odinger operator on the line graph $L(T)$ of $T$:
\[
R_T=\Delta-V,
\]
where $\Delta$ is a weighted graph Laplacian on $L(T)$ and $V$ is a diagonal potential. Explicitly,
\[
\Delta_{e,e'} =
\begin{cases}
\dfrac{1}{d_x} + \dfrac{1}{d_y} - 2, & e = e' = \{x,y\}, \\[10pt]
\dfrac{1}{d_z}, & e \cap e' = \{z\},\ e\neq e', \\[8pt]
0, & \text{otherwise},
\end{cases}
\qquad
V_{e,e'} =
\begin{cases}
\dfrac{2}{d_x} + \dfrac{2}{d_y} - 2, & e = e' = \{x,y\}, \\[8pt]
0, & e \neq e'.
\end{cases}
\]
A direct check shows that $\Delta$ has vanishing row sums, hence is a (negative semidefinite) weighted Laplacian: each diagonal entry equals the negative of the sum of the off-diagonal entries in its row, since an edge $e=\{x,y\}$ has exactly $d_x-1$ neighbours through $x$ (each contributing $1/d_x$) and $d_y-1$ through $y$ (each contributing $1/d_y$), giving row sum $(d_x-1)/d_x+(d_y-1)/d_y+\bigl(1/d_x+1/d_y-2\bigr)=0$. The potential $V$ is the diagonal multiplication operator $e\mapsto \tfrac{2}{d_x}+\tfrac{2}{d_y}-2$. Thus the spectral analysis of $R_T$, and in particular of $\lambda_{\max}(R_T)$, is the analysis of the top of the spectrum of a discrete Schr\"odinger operator $\Delta-V$ on the line graph.
\end{remark}

Its largest eigenvalue satisfies
\[
\lam(\RT)=\max_{\|f\|=1}\langle f,\RT f\rangle.
\]

Fix a vertex $v\in V(T)$. For a function $f:E(T)\to\mathbb R$, define
\[
S_v:=\sum_{e\ni v}f_e,\qquad
A_v:=\sum_{e\ni v}f_e^2,\qquad
\rho_v:=\frac{S_v^2}{A_v}.
\]
If $d_T(v)=d$, then by the Cauchy-Schwarz inequality, we have $1\le \rho_v\le d$.

\section{Monotonicity for One-Step Leaf Addition}\label{sec:one-step}

Before analyzing the general case, consider  starting from the single edge and repeatedly attaching new leaves at the 
one vertex, this forms the star graph $S_n$ with $n\ge 2$ vertices 
($|E| = n-1$ edges). A direct computation shows that the Perron eigenvalue is
$
\lambda_{\max}(R_{S_n}) = -\frac{2}{n-1} = -\frac{2}{|E|}
$, which strictly increase (from $-2$ toward $0$).  

\begin{lemma}\label{lem:mu-gt--1}
For any tree $T$ with at least two edges, $\lambda_{\max}(R_T) \ge  -1$.
\end{lemma}
\begin{proof}
Any tree $T$ with at least two edges, the maximum degree $\mathcal D\ge 2$. Applying Propostion 4 of \cite{BaiChengHua2026}, 
 $\lambda_{\max}(R_T) \ge -\frac{2}{\mathcal D} \ge -1$.
\end{proof}

Now, we consider any tree $T$ with at least two edges. Let $T'$ be obtained from $T$ by attaching one new pendant edge at vertex $v$, and let $d=d_T(v)$. Let $f: E(T) \to \mathbb{R}$ be an edge function on $T$. When we attach a new pendant edge $e_{\text{new}} = \{v, u\}$ at vertex $v$ (where $u$ is a new vertex) to obtain $T'$, we extend $f$ to a function $f_y$ on $T'$ as follows:
\[
f_y(e) = 
\begin{cases}
f(e), & e \in E(T) \subset E(T'), \\[4pt]
y, & e = e_{\text{new}},
\end{cases}
\]
where $y \in \mathbb{R}$ is a free parameter representing the value assigned to the new edge. Thus $f_y$ coincides with $f$ on all old edges and takes the constant value $y$ on the single new pendant edge. We use the vertex decomposition of the quadratic form (see \cite{ChengHua2026}, Proposition~3.1):
\[
\langle f, R_T f \rangle = \sum_{w \in V(T)} \frac{1}{d_T(w)} \left( S_w(f)^2 - 2A_w(f) \right),
\]
where
\[
S_w(f) = \sum_{e \ni w} f_e, \qquad A_w(f) = \sum_{e \ni w} f_e^2.
\]

When passing from $T$ to $T'$, only two vertices experience changes:
\begin{itemize}
\item Vertex $v$: its degree increases from $d$ to $d+1$;
\item The new leaf $u$: a new vertex of degree $1$ appears.
\end{itemize}
All other vertices retain their degrees and incident edge sets, hence their local contributions remain unchanged.

Finally, we have
\begin{proposition}[Rayleigh Difference Formula for One-Step Addition]\label{prop:delta}\cite{ChengHua2026}
\[
\langle f_y,R_{T'}f_y\rangle-\langle f,R_Tf\rangle
=
-\frac{S^2-2A}{d(d+1)}
+\frac{2Sy}{d+1}
-\frac{d+2}{d+1}y^2.
\]
\end{proposition}

The next proposition gives a sufficient condition for the existence of an extension $f_y$ of a unit vector $f$ on $T$ such that the Rayleigh quotient does not decrease.
\begin{proposition}[$\rho_v$-sharp criterion]\label{prop:sharp}
Let $f$ be a unit vector on $T$ and set $\mu=\langle f,R_Tf\rangle$. If
\[
2\bigl(d+2+\mu(d+1)\bigr)\ge \rho_v\bigl(2+\mu(d+1)\bigr),
\]
then there exists $y\in\mathbb R$ such that
\[
\frac{\langle f_y,R_{T'}f_y\rangle}{\langle f_y,f_y\rangle}\ge \mu.
\]
\end{proposition}

\begin{proof}
By Proposition \ref{prop:delta}, letting $\Delta(y):=\langle f_y,R_{T'}f_y\rangle-\langle f,R_Tf\rangle$, we have
\[
\Delta(y)
=
-\frac{S^2-2A}{d(d+1)}
+\frac{2Sy}{d+1}
-\frac{d+2}{d+1}y^2.
\]
Since $\|f\|=1$, $\langle f_y,f_y\rangle=1+y^2$, thus
\[
\frac{\langle f_y,R_{T'}f_y\rangle}{\langle f_y,f_y\rangle}-\mu
=
\frac{\Delta(y)-\mu y^2}{1+y^2}.
\]
Hence it suffices to show there exists $y$ such that $\Delta(y)-\mu y^2\ge0$.

Expanding,
\[
\Delta(y)-\mu y^2
=
-\frac{S^2-2A}{d(d+1)}
+\frac{2Sy}{d+1}
-
\left(\frac{d+2}{d+1}+\mu\right)y^2.
\]
This is a concave quadratic in $y$ as $\mu\ge -1$. Its maximum value is
\[
\max_y\bigl(\Delta(y)-\mu y^2\bigr)
=
\frac{2A}{d(d+1)}
-
\frac{(2+\mu(d+1))S^2}
{d(d+1)\bigl(d+2+\mu(d+1)\bigr)}.
\]
Substituting $S^2=\rho_v A$ gives
\[
\max_y\bigl(\Delta(y)-\mu y^2\bigr)
=
\frac{A}{d(d+1)}
\cdot
\frac{2\bigl(d+2+\mu(d+1)\bigr)-\rho_v\bigl(2+\mu(d+1)\bigr)}
{d+2+\mu(d+1)}.
\]
Under the hypothesis of the proposition, the right-hand side is nonnegative, so there exists $y$ with $\Delta(y)-\mu y^2\ge0$. Consequently,
\[
\frac{\langle f_y,R_{T'}f_y\rangle}{\langle f_y,f_y\rangle}\ge \mu.
\]
\end{proof}

\begin{corollary}[Implication for the Largest Eigenvalue]\label{cor:one-step}
Let $f$ be the unit Perron vector of $R_T$ corresponding to $\lam(R_T)$. If
\[
2\bigl(d+2+\lam(R_T)(d+1)\bigr)
\ge
\rho_v\bigl(2+\lam(R_T)(d+1)\bigr),
\]
then $\lam(R_{T'})\ge \lam(R_T)$.
\end{corollary}

\begin{proof}
Apply Proposition \ref{prop:sharp} with $\mu=\lam(R_T)$. It yields a $y$ such that
\[
\frac{\langle f_y,R_{T'}f_y\rangle}{\langle f_y,f_y\rangle}\ge \lam(R_T).
\]
By the Rayleigh principle,
\[
\lam(R_{T'})
\ge
\frac{\langle f_y,R_{T'}f_y\rangle}{\langle f_y,f_y\rangle}.
\]
Thus $\lam(R_{T'})\ge \lam(R_T)$. \end{proof}

\begin{remark}[Degree-dependent coarse threshold]\label{rem:theta}
Since $1\le \rho_v\le d$, the hardest case to satisfy in Proposition \ref{prop:sharp} is $\rho_v=d$. Substituting this into the sharp criterion
\[
2\bigl(d+2+\mu(d+1)\bigr)\ge \rho_v\bigl(2+\mu(d+1)\bigr)
\]
yields a sufficient condition depending only on $d$:
\[
\mu\le \theta(d),
\]
where
\[
\theta(d)=
\begin{cases}
+\infty, & d=1,2,\\[4pt]
\dfrac{4}{(d+1)(d-2)}, & d\ge3.
\end{cases}
\]
Therefore, if $\lam(R_T)\le \theta(d)$, then attaching a pendant edge at vertex $v$ necessarily satisfies $\lam(R_{T'})\ge \lam(R_T)$. This $\theta(d)$ is kept as it provides a direct degree-based uniform threshold, compressing the local information $\rho_v$.
\end{remark}

\section{Limit of Repeated Leaf Addition at a Fixed Vertex}\label{sec:limit}

Fix a tree $T$ and a vertex $v$ on it. Let $d=d_T(v)$. For each $k\ge0$, let $T_k$ be the tree obtained by attaching $k$ additional pendant edges to $v$, and define $\lambda_k:=\lam(R_{T_k})$.

Let the neighbors of $v$ in the original tree be $u_1,\dots,u_d$. Removing $v$ yields connected components $C_1,\dots,C_d$, with $u_j\in C_j$. Denote the $k$ new pendant edges by $g_1,\dots,g_k$.

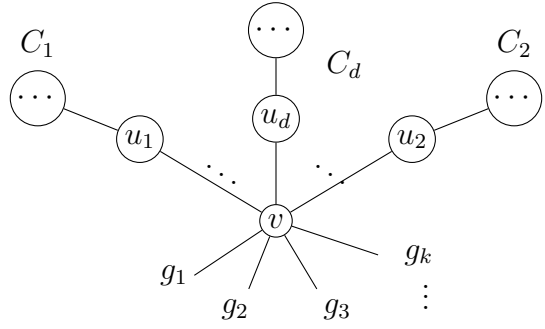
\begin{figure}[h]
\centering
\begin{tikzpicture}[
    scale=0.9,
    every node/.style={circle,draw,inner sep=1.5pt},
    >=stealth
]

\node (v) at (0,0) {$v$};

\node (u1) at (-2,1.2) {$u_1$};
\node (c1) at (-3.5,1.8) {$\cdots$};
\draw (v) -- (u1);
\draw (u1) -- (c1);
\node[draw=none] at (-3.5,2.6) {$C_1$};

\node (u2) at (2,1.2) {$u_2$};
\node (c2) at (3.5,1.8) {$\cdots$};
\draw (v) -- (u2);
\draw (u2) -- (c2);
\node[draw=none] at (3.5,2.6) {$C_2$};

\node (ud) at (0,1.5) {$u_d$};
\node (cd) at (0,2.8) {$\cdots$};
\draw (v) -- (ud);
\draw (ud) -- (cd);
\node[draw=none] at (1,2.3) {$C_d$};

\node[draw=none] at (-0.8,0.8) {$\ddots$};
\node[draw=none] at (0.8,0.8) {$\ddots$};

\draw (v) -- (-1.2,-0.8);
\draw (v) -- (-0.4,-1);
\draw (v) -- (0.6,-1);
\draw (v) -- (1.5,-0.5);
\node[draw=none] at (-1.5,-0.8) {$g_1$};
\node[draw=none] at (-0.6,-1.3) {$g_2$};
\node[draw=none] at (0.9,-1.3) {$g_3$};
\node[draw=none] at (2.1,-0.5) {$g_k$};
\node[draw=none] at (2.2,-1) {$\vdots$};


\end{tikzpicture}
\caption{The tree $T_k$: original branches $C_1,\dots,C_d$ attached to $v$ from above, 
with $k$ new pendant edges $g_1,\dots,g_k$ attached below $v$.}
\label{fig:Tk}
\end{figure}
Since any permutation of $g_1,\dots,g_k$ is an automorphism commuting with $R_{T_k}$, and the Perron eigenvector is the unique positive eigenvector, it must take the same value on all these new pendant edges. Hence, to compute the Perron eigenvalue, we only need to retain coordinates for each old edge and combine the new pendant edges into a single common coordinate $y$. On this invariant subspace, the restriction of $R_{T_k}$ is described by a fixed-dimensional representation matrix $Q_k$; thus $Q_k$ and the restricted operator share eigenvalues, and in particular,
\[
\lambda_k=\lambda_{\max}(Q_k),  \ \ \forall k\ge 1.
\]

We now detail the structure of $Q_k$. Set $d_k = d+k$. If an old edge is not incident to $v$, its matrix entries are independent of $k$. For an edge $e_j=\{v,u_j\}$ (the root edge of the $j$-th branch),
\[
(Q_k)_{e_j,e_j}
=
-\left(\frac1{d_T(u_j)}+\frac1{d_k}\right).
\]
Thus, within the same branch, corrections originating from vertex $v$ carry a factor $1/d_k$. For two old edges from different branches, any coupling occurs only through the common vertex $v$, so such blocks also have an overall $1/d_k$ factor. Coupling between a branch and the new leaf orbit distinguishes directions: from the old edge equation, the contribution of the $k$ identical new leaf edges sums, yielding a factor $k/d_k$; from the new leaf equation, the coupling back to a specific old edge remains $1/d_k$. The diagonal entry for the new leaf orbit itself equals $-(d+2)/d_k$.

Ordering basis vectors as (coordinates of old edges in $C_1$, $\dots$, $C_d$, $y$), the matrix $Q_k$ has the block form
\[
Q_k=
\begin{pmatrix}
A_1+\dfrac{1}{d+k}B_{11} & \dfrac{1}{d+k}B_{12} & \cdots & \dfrac{1}{d+k}B_{1d} & c_1+\dfrac{1}{d+k}b_1 \\
\dfrac{1}{d+k}B_{21} & A_2+\dfrac{1}{d+k}B_{22} & \cdots & \dfrac{1}{d+k}B_{2d} & c_2+\dfrac{1}{d+k}b_2 \\
\vdots & \vdots & \ddots & \vdots & \vdots \\
\dfrac{1}{d+k}B_{d1} & \dfrac{1}{d+k}B_{d2} & \cdots & A_d+\dfrac{1}{d+k}B_{dd} & c_d+\dfrac{1}{d+k}b_d \\
\dfrac{1}{d+k}\beta_1^{\mathsf T} & \dfrac{1}{d+k}\beta_2^{\mathsf T} & \cdots & \dfrac{1}{d+k}\beta_d^{\mathsf T} & -\dfrac{d+2}{d+k}
\end{pmatrix},
\]
where $A_j, B_{ij}, b_j, c_j, \beta_j$ are matrices/vectors independent of $k$.

%
%

\medskip
\noindent\emph{The diagonal blocks $A_j$.}
$A_j$ is obtained from the formula for $R_T$ by setting $1/d_v$ to zero everywhere it appears; equivalently, by treating $v$ as having infinite degree, or by imposing a Dirichlet boundary condition at $v$. In particular:
\begin{itemize}[leftmargin=2em]
\item the diagonal entry of $A_j$ at the root edge $e_j=\{v,u_j\}$ is $-1/d_T(u_j)$ (not $-(1/d_T(u_j)+1/d_T(v))$);

\item all entries of $A_j$ that do not involve $v$ at all are unchanged from $R_T$.
\end{itemize}
This is exactly the limit of the corresponding block of $Q_k$ as $k\to\infty$ (since $d_k=d+k\to\infty$, every $1/d_k$ factor disappears).

\medskip
\noindent\emph{The off-diagonal blocks $B_{ij}$.}
For $i\neq j$, the block $B_{ij}$ records the $1/(d+k)$ coupling between branches $i$ and $j$ that occurs only through the shared vertex $v$. Since any such coupling involves both root edges and the factor $1/d_v$, the block $B_{ij}$ is supported only on the $(e_i,e_j)$ entry, and there it equals $1$. Equivalently, $B_{ij}=c_i c_j^{\mathsf T}$ in terms of the column vectors introduced below. For $i=j$, the block $B_{jj}$ records the within-branch $1/(d+k)$ correction at the diagonal of $e_j$; it equals $-c_j c_j^{\mathsf T}$.

\medskip
\noindent\emph{The interface vectors $c_j$, $b_j$, $\beta_j$.}
The vector $c_j$ is the $k$-independent piece of the column coupling branch $j$ to the new-leaf coordinate $y$. The total $(e,y)$-entry of $Q_k$ for an edge $e$ of branch $j$ is
\[
(Q_k)_{e,y}
=
\begin{cases}
\dfrac{k}{d+k}, & e=e_j,\\[6pt]
0, & e\in E(C_j),
\end{cases}
\]
because only the root edge $e_j$ is incident to $v$ and hence to the new-leaf cluster. Writing $k/(d+k)=1-d/(d+k)$, we read off
\[
c_j = \mathbf{1}_{e_j},
\qquad
b_j = -d\,\mathbf{1}_{e_j} = -d\, c_j,
\]
where $\mathbf{1}_{e_j}$ denotes the indicator (standard basis) vector supported on the root edge $e_j$ in the branch-$j$ basis. Thus $c_j$ is \emph{not} an all-ones vector: it has exactly one nonzero entry, located at the root edge.

By symmetry of $R_{T_k}$, the corresponding row coupling $\beta_j^{\mathsf T}$ also acts only through the root edge:
\[
\beta_j = \mathbf{1}_{e_j} = c_j.
\]
The scalar diagonal entry of $B$ at the $(y,y)$ position is
\[
B_{yy} = -(d+2),
\]
which is the $k$-independent part of $-\dfrac{d+2}{d+k}$, the diagonal entry of $Q_k$ 
at the normalized leaf-cluster coordinate $y$.

\medskip
\noindent\emph{Consequence: structure of $Q_\infty$.}
Substituting the descriptions above and letting $k\to\infty$,
\[
Q_\infty
=
\begin{pmatrix}
A_1 &        &       & c_1 \\
    & \ddots &       & \vdots \\
    &        & A_d   & c_d \\
0   & \cdots & 0     & 0
\end{pmatrix},
\]
i.e., $Q_\infty$ is block upper-triangular with diagonal blocks $A_1,\dots,A_d$ and a final scalar block $0$. The off-diagonal blocks between distinct branches all vanish in the limit (they were $O(1/(d+k))$), and the bottom row vanishes entirely (the $(y,e)$ couplings in $Q_k$ are all $O(1/(d+k))$). Consequently,
\[
\sigma(Q_\infty) = \{0\} \cup \bigcup_{j=1}^{d} \sigma(A_j),
\]
where $\sigma(\cdot)$ denotes the spectrum (i.e., the set of eigenvalues, counted with multiplicity).
In particular
\[
\lambda_{\max}(Q_\infty)
=
\max\bigl(0,\ \lambda_{\max}(A_1),\dots,\lambda_{\max}(A_d)\bigr),
\]
which is the content of Proposition~\ref{prop:limit}.

\medskip
\noindent\emph{Verification on the example.}
For the base tree of Section~\ref{sec:example} with $d=d_T(v)=2$, the branch matrices are
\[
A_1=
\begin{pmatrix}
-\tfrac12 & \tfrac12\\[2pt]
\tfrac12 & -\tfrac32
\end{pmatrix},
\qquad
A_2=
\begin{pmatrix}
-\tfrac13 & \tfrac23\\[2pt]
\tfrac13 & -1
\end{pmatrix}.
\]
In each case the upper-left diagonal entry is $-1/d_T(u_j)$ (namely $-1/2$ for $u_1$ and $-1/3$ for $u_2$), \emph{without} an additional $-1/d_T(v)=-1/2$ contribution. 

The interface vectors are
\[
c_1=\beta_1=\begin{pmatrix}1\\0\end{pmatrix},\qquad
c_2=\beta_2=\begin{pmatrix}1\\0\end{pmatrix},\qquad
b_1=\begin{pmatrix}-2\\0\end{pmatrix},\qquad
b_2=\begin{pmatrix}-2\\0\end{pmatrix},
\]
so that $b_j = -d\,c_j$ and $\beta_j = c_j$ hold in this example. One checks immediately that the matrices $Q_\infty$ and $B$ written out in Section~\ref{sec:example} are reproduced by these descriptions.

\begin{proposition}[Limit Formula]\label{prop:limit}
\[
\lambda_k\longrightarrow \lambda_\infty:=\lambda_{\max}(Q_\infty)
=
\max\Bigl(0,\lambda_{\max}(A_1),\dots,\lambda_{\max}(A_d)\Bigr).
\]
\end{proposition}

\begin{proof}
From the block representation, we can directly write
\[
Q_k=Q_\infty+\frac{1}{d+k}B,
\]
where
\[
B=
\begin{pmatrix}
B_{11} & B_{12} & \cdots & B_{1d} & b_1 \\
B_{21} & B_{22} & \cdots & B_{2d} & b_2 \\
\vdots & \vdots & \ddots & \vdots & \vdots \\
B_{d1} & B_{d2} & \cdots & B_{dd} & b_d \\
\beta_1^{\mathsf T} & \beta_2^{\mathsf T} & \cdots & \beta_d^{\mathsf T} & -(d+2)
\end{pmatrix}
\]
is independent of $k$. Thus $Q_k\longrightarrow Q_\infty$ as $k\to\infty$. Since the largest eigenvalue is continuous for matrices depending continuously on parameters,
\[
\lambda_k=\lambda_{\max}(Q_k)\longrightarrow \lambda_{\max}(Q_\infty).
\]

Furthermore, $Q_\infty$ is block upper-triangular with diagonal blocks $A_1,\dots,A_d,0$. Hence,
\[
\lambda_{\max}(Q_\infty)
=
\max\Bigl(0,\lambda_{\max}(A_1),\dots,\lambda_{\max}(A_d)\Bigr).
\]
This completes the proof.
\end{proof}

\begin{corollary}
\[
\lambda_\infty=0
\quad\Longleftrightarrow\quad
\lambda_{\max}(A_j)\le0\qquad (j=1,\dots,d).
\]
\end{corollary}

\begin{proof}
This is a direct restatement of
\[
\lambda_\infty
=
\max\Bigl(0,\lambda_{\max}(A_1),\dots,\lambda_{\max}(A_d)\Bigr)
\]
from Proposition \ref{prop:limit}.
\end{proof}

\section{Tail Monotonicity}\label{sec:monotonicity}
Using the notation from the previous section,
\[
Q_k=Q_\infty+\frac1{d+k}B.
\]
Set
\[
\varepsilon_k:=\frac1{d+k},
\qquad
Q(\varepsilon):=Q_\infty+\varepsilon B.
\]
Then $Q_k=Q(\varepsilon_k)$ and $\varepsilon_k\to0^+$.

\begin{proposition}[First-Order Asymptotic Expansion]\label{prop:asym}
Assume that $\lambda_\infty=\lambda_{\max}(Q_\infty)$ is a simple eigenvalue of $Q_\infty$. Let $r,\ell$ be corresponding right and left eigenvectors satisfying
\[
Q_\infty r=\lambda_\infty r,\qquad
\ell^{\mathsf T}Q_\infty=\lambda_\infty \ell^{\mathsf T},\qquad
\ell^{\mathsf T}r=1.
\]
Define $\alpha:=\ell^{\mathsf T}Br$. Then as $k\to\infty$,
\[
\lambda_k
=
\lambda_\infty+\frac{\alpha}{d+k}
+O\!\left(\frac1{(d+k)^2}\right).
\]
\end{proposition}

\begin{proof}
Since $\lambda_\infty$ is a simple eigenvalue of $Q_\infty$, analytic perturbation theory guarantees the existence of an analytic eigenvalue branch $\lambda(\varepsilon)$ and analytic eigenvector branches $r(\varepsilon),\ell(\varepsilon)$ in a neighborhood of $\varepsilon=0$ such that
\[
Q(\varepsilon)r(\varepsilon)=\lambda(\varepsilon)r(\varepsilon),
\]
\[
\ell(\varepsilon)^{\mathsf T}Q(\varepsilon)=\lambda(\varepsilon)\ell(\varepsilon)^{\mathsf T},
\]
with normalizations
\[
\ell(\varepsilon)^{\mathsf T}r(\varepsilon)=1,
\qquad
\lambda(0)=\lambda_\infty,\ 
r(0)=r,\ 
\ell(0)=\ell.
\]

Differentiating $Q(\varepsilon)r(\varepsilon)=\lambda(\varepsilon)r(\varepsilon)$ at $\varepsilon=0$ yields
\[
Br+Q_\infty r'(0)=\lambda'(0)r+\lambda_\infty r'(0).
\]
Multiplying on the left by $\ell^{\mathsf T}$ and using $\ell^{\mathsf T}Q_\infty=\lambda_\infty \ell^{\mathsf T}$ and $\ell^{\mathsf T}r=1$ gives $\lambda'(0)=\ell^{\mathsf T}Br=\alpha$. Hence, the Taylor expansion of $\lambda(\varepsilon)$ around $\varepsilon=0$ is
\[
\lambda(\varepsilon)
=
\lambda_\infty+\alpha\varepsilon+O(\varepsilon^2).
\]

Because $\lambda_\infty$ is simple at the limit, for sufficiently small $\varepsilon>0$, this analytic branch corresponds to the largest eigenvalue of $Q(\varepsilon)$. Thus for sufficiently large $k$,
\[
\lambda_k=\lambda_{\max}(Q_k)=\lambda(\varepsilon_k).
\]
Substituting $\varepsilon_k=1/(d+k)$ yields the desired expansion.
\end{proof}

\begin{theorem}[Tail Monotonicity]\label{thm:tail-degenerate}
Under the assumptions of Proposition \ref{prop:asym}, if $\alpha\neq0$, then the sequence $\{\lambda_k\}$ is strictly monotonic for sufficiently large $k$. Specifically:
\begin{enumerate}[label=\rm(\arabic*)]
\item If $\alpha>0$, then for sufficiently large $k$,
\[
\lambda_k>\lambda_\infty,\qquad
\lambda_{k+1}<\lambda_k;
\]
\item If $\alpha<0$, then for sufficiently large $k$,
\[
\lambda_k<\lambda_\infty,\qquad
\lambda_{k+1}>\lambda_k.
\]
\end{enumerate}
\end{theorem}

\begin{proof}
From Proposition \ref{prop:asym},
\[
\lambda_k-\lambda_\infty
=
\frac{\alpha}{d+k}
+O\!\left(\frac1{(d+k)^2}\right).
\]
Hence for sufficiently large $k$, $\lambda_k-\lambda_\infty$ has the same sign as $\alpha$, yielding
\[
\alpha>0 \Rightarrow \lambda_k>\lambda_\infty,
\qquad
\alpha<0 \Rightarrow \lambda_k<\lambda_\infty.
\]

Now prove eventual monotonicity. As shown in the proof of Proposition \ref{prop:asym}, $\lambda_k=\lambda(\varepsilon_k)$, where $\lambda(\varepsilon)$ is analytic near $\varepsilon=0$ and $\lambda'(0)=\alpha\neq0$. Thus for sufficiently small $\varepsilon$, $\lambda'(\varepsilon)$ has the same sign as $\alpha$. That is, there exists $\varepsilon_0>0$ such that
\[
0<\varepsilon<\varepsilon_0
\quad\Longrightarrow\quad
\operatorname{sgn}\lambda'(\varepsilon)=\operatorname{sgn}\alpha.
\]
Since $\varepsilon_{k+1}<\varepsilon_k$ and $\varepsilon_k\to0$, for sufficiently large $k$ we have $\varepsilon_k,\varepsilon_{k+1}\in(0,\varepsilon_0)$. Therefore:

If $\alpha>0$, $\lambda(\varepsilon)$ is strictly increasing on $(0,\varepsilon_0)$, and $\varepsilon_{k+1}<\varepsilon_k$, so
\[
\lambda_{k+1}=\lambda(\varepsilon_{k+1})<\lambda(\varepsilon_k)=\lambda_k.
\]

If $\alpha<0$, $\lambda(\varepsilon)$ is strictly decreasing on $(0,\varepsilon_0)$, and $\varepsilon_{k+1}<\varepsilon_k$, so
\[
\lambda_{k+1}=\lambda(\varepsilon_{k+1})>\lambda(\varepsilon_k)=\lambda_k.
\]

Combining both cases yields the conclusion.
\end{proof}

\begin{remark}
Proposition~\ref{prop:asym} assumes $\lambda_\infty$ is simple. If $\lambda_\infty$ is multiple
(e.g., maximal eigenvalues from different branches coincide or coincide with the scalar block $0$),
standard degenerate perturbation theory \cite{Kato1995} gives the first-order asymptotic formula
\[
\lambda_k = \lambda_\infty + \frac{\alpha_{\max}}{d+k} + O\!\left(\frac1{(d+k)^2}\right),
\]
where $\alpha_{\max} = \lambda_{\max}(W)$ and $W$ is the compression of $B$ to the eigenspace of $\lambda_\infty$.
The conclusion regarding tail monotonicity remains unchanged. The detailed derivation for the multiple case can be found in Chapter II  of \cite{Kato1995} and is omitted here for brevity.
\end{remark}

\section{Example}\label{sec:example}

\subsection{A case with $\alpha>0$}
Consider the base tree $T$ as shown: vertex $v$ has a chain of length $2$ to its left and a small binary fork to its right. Pendant edges are repeatedly added at $v$.

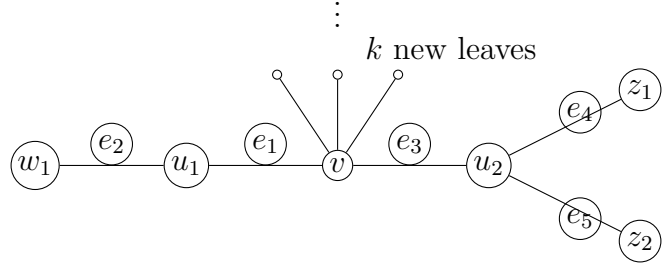
\begin{figure}[htp]
    \centering
\begin{tikzpicture}[scale=1, every node/.style={circle,draw,inner sep=1.2pt}]
\node (v) at (0,0) {$v$};
\node (u1) at (-2,0) {$u_1$};
\node (w1) at (-4,0) {$w_1$};
\node (u2) at (2,0) {$u_2$};
\node (z1) at (4,1) {$z_1$};
\node (z2) at (4,-1) {$z_2$};
\node[draw=none] (dots) at (0,2.1) {$\vdots$};
\node (t1) at (0,1.2) {};
\node (t2) at (-0.8,1.2) {};
\node (t3) at (0.8,1.2) {};
\draw (v)--node[above] {$e_1$} (u1);
\draw (u1)--node[above] {$e_2$} (w1);
\draw (v)--node[above] {$e_3$} (u2);
\draw (u2)--node[above right] {$e_4$} (z1);
\draw (u2)--node[below right] {$e_5$} (z2);
\draw (v)--(t1);
\draw (v)--(t2);
\draw (v)--(t3);
\node[draw=none] at (1.5,1.55) {$k$ new leaves};
\end{tikzpicture}
\caption{A schematic illustration of repeated pendant-edge attachment at the fixed vertex \(v\).}
\label{fig:repeated-leaf-attachment}
\end{figure}

Here
\[
d_T(v)=2,\qquad d_T(u_1)=2,\qquad d_T(u_2)=3.
\]
Let the $k$ new pendant edges in $T_k$ share a common coordinate $y$; denote the coordinates of the old edges as
\[
a=f(e_1),\quad b=f(e_2),\quad c=f(e_3),\quad d=f(e_4)=f(e_5).
\]
The Perron eigenvector equation becomes
\[
\lambda
\begin{pmatrix}
a\\ b\\ c\\ d\\ y
\end{pmatrix}
=
Q_k
\begin{pmatrix}
a\\ b\\ c\\ d\\ y
\end{pmatrix},
\]
where
\[
Q_k=
\begin{pmatrix}
-\dfrac12-\dfrac1{k+2} & \dfrac12 & \dfrac1{k+2} & 0 & \dfrac{k}{k+2} \\[8pt]
\dfrac12 & -\dfrac32 & 0 & 0 & 0 \\[8pt]
\dfrac1{k+2} & 0 & -\dfrac13-\dfrac1{k+2} & \dfrac23 & \dfrac{k}{k+2} \\[8pt]
0 & 0 & \dfrac13 & -1 & 0 \\[8pt]
\dfrac1{k+2} & 0 & \dfrac1{k+2} & 0 & -\dfrac4{k+2}
\end{pmatrix}.
\]

This matrix illustrates the structure discussed earlier:
\begin{enumerate}[label=\rm(\arabic*)]
\item The upper-left $2\times2$ block
\[
A_1=
\begin{pmatrix}
-\dfrac12 & \dfrac12\\[4pt]
\dfrac12 & -\dfrac32
\end{pmatrix},
\]
and the next $2\times2$ block
\[
A_2=
\begin{pmatrix}
-\dfrac13 & \dfrac23\\[4pt]
\dfrac13 & -1
\end{pmatrix}
\]
correspond to the two branches after deleting $v$;
\item The entry $k/(k+2) = 1 - 2/(k+2)$ in the fifth column shows the constant main term from the feedback of the whole leaf cluster on the old edge equations;
\item The fifth row remains of order $1/(k+2)$, representing the coupling from a single new leaf equation back to the old edges.
\end{enumerate}

Rewrite as $Q_k=Q_\infty+\frac1{k+2}B$, where
\[
Q_\infty=
\begin{pmatrix}
-\dfrac12 & \dfrac12 & 0 & 0 & 1 \\[6pt]
\dfrac12 & -\dfrac32 & 0 & 0 & 0 \\[6pt]
0 & 0 & -\dfrac13 & \dfrac23 & 1 \\[6pt]
0 & 0 & \dfrac13 & -1 & 0 \\[6pt]
0 & 0 & 0 & 0 & 0
\end{pmatrix},
\qquad
B=
\begin{pmatrix}
-1 & 0 & 1 & 0 & -2 \\[4pt]
0 & 0 & 0 & 0 & 0 \\[4pt]
1 & 0 & -1 & 0 & -2 \\[4pt]
0 & 0 & 0 & 0 & 0 \\[4pt]
1 & 0 & 1 & 0 & -4
\end{pmatrix}.
\]
This matches the affine perturbation form from Propositions \ref{prop:limit} and \ref{prop:asym}.

Now compute the limit blocks. For the two $2\times2$ matrices:
\[
\lambda_{\max}(A_1)=-1+\frac1{\sqrt2}<0,
\qquad
\lambda_{\max}(A_2)=\frac{-2+\sqrt3}{3}<0.
\]
Hence,
\[
\lambda_\infty
=
\max\bigl(0,\lambda_{\max}(A_1),\lambda_{\max}(A_2)\bigr)
=
0,
\]
consistent with Proposition \ref{prop:limit}.

Check the first-order term. For eigenvalue $0$, a right eigenvector is
\[
r=
\begin{pmatrix}
3\\ 1\\ 9\\ 3\\ 1
\end{pmatrix},
\qquad
Q_\infty r=0,
\]
and a corresponding left eigenvector is
\[
\ell^{\mathsf T}=
\begin{pmatrix}
0&0&0&0&1
\end{pmatrix},
\qquad
\ell^{\mathsf T}Q_\infty=0,
\qquad
\ell^{\mathsf T}r=1.
\]
Then
\[
\alpha=\ell^{\mathsf T}Br
=
\begin{pmatrix}
1&0&1&0&-4
\end{pmatrix}
\begin{pmatrix}
3\\ 1\\ 9\\ 3\\ 1
\end{pmatrix}
=8>0.
\]
Thus Proposition \ref{prop:asym} gives
\[
\lambda_k=\frac8{k+2}+O\!\left(\frac1{(k+2)^2}\right),
\]
and the eventual tail monotonicity theorem implies that for sufficiently large $k$,
\[
\lambda_k>0,\qquad \lambda_{k+1}<\lambda_k.
\]

Finally, numerical values for comparison with theory:
\[
\begin{array}{c|ccccccccc}
k & 0 & 1 & 2 & 3 & 5 & 10 & 20 & 50 & 100\\
\hline
\lambda_k
& -0.1731 & -0.0312 & 0.0310 & 0.0628 & 0.0906
& 0.1028 & 0.0922 & 0.0643 & 0.0436
\end{array}
\]
Observations:
\begin{enumerate}[label=\rm(\arabic*)]
\item The early terms are not monotonic, even transitioning from negative to positive values;
\item The limit is indeed $0$;
\item From the middle range onward, the values decrease towards $0$;
\item This illustrates exactly: Proposition \ref{prop:limit} governs the limit, Proposition \ref{prop:asym} and the eventual tail monotonicity theorem govern the tail behavior.
\end{enumerate}

\subsection{A case with $\alpha<0$}

Consider the simplest possible base tree $T$: a single edge $e_1=\{v,u_1\}$ with $u_1$ a leaf. Then $d=d_T(v)=1$, and $T_k$ is a star $K_{1,k+1}$. One computes directly
\[
\lambda_k = \lambda_{\max}(R_{K_{1,k+1}}) = -\frac{2}{k+1}.
\]
Hence $\lambda_\infty = 0$, and expanding
\[
\lambda_k = -\frac{2}{1+k} = \frac{\alpha}{1+k}
\]
gives $\alpha = -2 < 0$. The sequence $\lambda_k$ approaches $0$ from below and is strictly increasing. This matches the prediction of Theorem~\ref{thm:tail-degenerate} with $\alpha<0$.

\medskip
{\bf Acknowledgements:} 
S. Bai is supported by NSFC, no.12301434. B. Hua is supported by NSFC, no.12371056.

\bibliographystyle{plain}  
\bibliography{references}
\end{document}